\documentclass[12pt]{article}
\usepackage{amsfonts, amsmath, amssymb, latexsym, eucal, amscd} 

\usepackage{amsthm}
\usepackage{amscd}

\usepackage{cite}

\newtheorem{definition}{Definition}[section]
\newtheorem{theorem}[definition]{Theorem}
\newtheorem{lemma}[definition]{Lemma}
\newtheorem{corollary}[definition]{Corollary}
\newtheorem{remark}[definition]{Remark}

\newtheorem{proposition}[definition]{Proposition}

\begin{document} 

\title{\bf A $Q$-polynomial structure for the\\  Attenuated Space poset $\mathcal A_q(N,M)$
}
\author{
Paul Terwilliger}
\date{}

\maketitle
\begin{abstract} The goal of this article is to display a $Q$-polynomial structure for the Attenuated Space poset $\mathcal A_q(N,M)$.
The poset   $\mathcal A_q(N,M)$ is briefly described as follows.  Start with an $(N+M)$-dimensional vector space $H$ over a finite field with $q$ elements. Fix an $M$-dimensional subspace $h$ of $H$.
The vertex set $X$ of $\mathcal A_q(N,M)$ consists of the subspaces of $H$ that have zero intersection with $h$. The partial order on $X$  is the inclusion relation.
The  $Q$-polynomial structure involves two matrices $A, A^* \in {\rm Mat}_X(\mathbb C)$ with the following entries.
For $y, z \in X$ the matrix $A$ has $(y,z)$-entry $1$ (if $y$ covers $z$); $q^{{\rm dim}\,y}$ (if $z$ covers $y$); and 0 (if neither of  $y,z$ covers the other).
The matrix $A^*$ is diagonal, with $(y,y)$-entry $q^{-{\rm dim}\,y}$ for all $y\in X$. By construction, $A^*$ has $N+1$ eigenspaces. By construction, $A$
acts on these eigenspaces in a (block) tridiagonal fashion.
We show that $A$ is diagonalizable, with $2N+1$ eigenspaces.
We show that $A^*$ acts on these eigenspaces  in
a (block) tridiagonal fashion. Using this action, we show that $A$ is $Q$-polynomial. We show that $A, A^*$ satisfy a pair of relations called the tridiagonal relations.
We consider the subalgebra $T$ of ${\rm Mat}_X(\mathbb C)$ generated by $A, A^*$. 
 We show that $A,A^*$ act on each irreducible $T$-module as a Leonard pair.

\bigskip

 \end{abstract}
 
 \section{Introduction}
 In the subject of Algebraic Combinatorics, there is a type of finite, undirected, connected graph, said to be distance-regular 
 \cite{bannai, bbit, bcn, dkt}. Many distance-regular graphs enjoy an algebraic property called $Q$-polynomial. This property
 was introduced by Delsarte \cite{delsarte} and further investigated in \cite{bannai, bbit, bcn}. Some surveys about the
 $Q$-polynomial property can be found in \cite{dkt, int}.
 \medskip
 
 \noindent In \cite[Section~20]{int} and  \cite[Section~2]{Lnq} we extended the $Q$-polynomial property to graphs that are not necessarily distance-regular. 
 In \cite[Section~3]{Lnq} we considered a partially ordered set called the projective geometry $L_N(q)$. We extended the $Q$-polynomial property to 
 $L_N(q)$, by viewing its Hasse diagram as an undirected graph. In \cite{Lnq} we showed that $L_N(q)$ has a $Q$-polynomial structure, and we described
 this structure in detail.
 \medskip
 
 \noindent The poset $L_N(q)$ has a relative called the Attenuated Space poset $\mathcal A_q(N,M)$; see
 \cite[Section~9.5]{bcn} and \cite{WenLiu}. We now summarize some results about $\mathcal A_q(N,M)$. 
 In \cite{delsarte2} Delsarte showed that $\mathcal A_q(N,M)$ is a regular semilattice.
 In \cite{sprague}  Sprague characterized $\mathcal A_q(N,M)$ in terms of diagram geometries.
 In \cite{THuang} Huang used $\mathcal A_q(N,M)$ to characterize the association scheme of the bilinear forms.
 In \cite{uniform} we showed that $\mathcal A_q(N,M)$ is uniform in the sense of \cite[Definition~2.2]{uniform}.
 In \cite[Section~40]{quantumMatroid}  we showed that $\mathcal A_q(N,M)$ is a regular quantum matroid.
 In \cite{bonoli}
 Bonoli and Melone  gave a geometrical
characterization of $\mathcal A_q(N,M)$.
In \cite{tanaka} Tanaka used $\mathcal A_q(N,M)$ to describe the subsets of minimal width and dual width in the association scheme of the bilinear forms.
In \cite{guo, wang}
Wang, Guo, and Li constructed an association
scheme based on $\mathcal A_q(N,M)$.
They found its intersection numbers and
investigated the incidence matrices.
The character table of this association scheme was found by 
Kurihara \cite{kurihara}.
In \cite{terAugDU} the present author and Worawannotai used $\mathcal A_q(N,M)$ to obtain a representation of an augmented down-up algebra.
In \cite{gao} 
Gao and Wang used $\mathcal A_q(N,M)$ to obtain some error-correcting codes.
In \cite{LiuWang} 
 Liu and Wang found the automorphism group of some graphs based
on $\mathcal A_q(N,M)$.
\medskip

\noindent 
The poset   $\mathcal A_q(N,M)$ is briefly described as follows (formal definitions begin in Section 2).  Start with an $(N+M)$-dimensional vector space $H$ over a finite field with $q$ elements. Fix an $M$-dimensional subspace $h$ of $H$.
The vertex set $X$ of $\mathcal A_q(N,M)$ consists of the subspaces of $H$ that have zero intersection with $h$. The partial order on $X$  is the inclusion relation.
The poset $\mathcal A_q(N,M)$ is ranked with rank $N$; the rank of a vertex is equal to its dimension.
\medskip

\noindent  The main goal of this article
 is to display a $Q$-polynomial
structure for $\mathcal A_q(N,M)$. We now summarize our results.
We define two matrices $A, A^* \in {\rm Mat}_X(\mathbb C)$ with the following entries.
For $y, z \in X$ the matrix $A$ has $(y,z)$-entry $1$ (if $y$ covers $z$); $q^{{\rm dim}\,y}$ (if $z$ covers $y$); and 0 (if neither of  $y,z$ covers the other).
The matrix $A^*$ is diagonal, with $(y,y)$-entry $q^{-{\rm dim}\,y}$ for all $y\in X$. By construction, $A^*$ has $N+1$  eigenvalues $\lbrace q^{-i} \rbrace_{i=0}^N$. 
For $0 \leq i \leq N$ the eigenspace for $q^{-i}$ is called the $i$th subconstituent of $\mathcal A_q(N,M)$.
 By construction, $A$
acts on these subconstituents in a (block) tridiagonal fashion.
We show that $A$ is diagonalizable, with  $2N+1$ eigenspaces. 
We show that $A^*$ acts on these eigenspaces  in
a (block) tridiagonal fashion. Using the $A^*$ action, we show that $A$ is $Q$-polynomial. This is our main result; see Theorem  \ref{thm:mainRes} below.
The following subsidiary results may be of independent interest.
For each eigenspace of $A$ we compute the dimension (Theorem \ref{thm:dimEV}) and corresponding eigenvalue (Theorem \ref{thm:AeigVal}).
We show that $A, A^*$ satisfy a pair of relations called the tridiagonal relations (Propositions \ref{prop:TD1}, \ref{cor:TD2}). To obtain these relations, we write 
$A=R+L$ where $R$ raises the subconstituents and $L$ lowers the subconstituents. By construction, $RA^*=q A^* R$ and $LA^* = q^{-1} A^* L$. We show that $R, L$ satisfy a pair of relations called the down-up relations (Proposition \ref{lem:RLAs}).
We consider the subalgebra $T$ of ${\rm Mat}_X(\mathbb C)$ generated by $A, A^*$. 
 We show that $A,A^*$ act on each irreducible $T$-module as a Leonard pair (Section 11).
\medskip

\noindent We call $A$ the $q$-adjacency matrix of $\mathcal A_q(N,M)$. We would like to acknowledge that $A$ is motivated by the breakthrough article \cite{murali} coauthored by Ghosh and Srinivasan.
\medskip

\noindent The paper is organized as follows. Section 2 contains some preliminaries.
 In Section 3, we define the poset $\mathcal A_q(N,M)$ and give some basic facts about it.
In Section 4, we introduce the matrix $A^*$ and investigate the algebra it generates.
In Section 5, we introduce the matrix $R$  and a variation $L'$ of the matrix $L$. We describe how $R, L', A^*$ are related.
In Section 6, we introduce the algebra $T$ and describe the irreducible $T$-modules. We consider how $R, L', A^*$ act on each irreducible $T$-module.
In Section 7, we introduce the matrix $L$ and describe how $R, L, A^*$ are related. We consider how $L$ acts on each irreducible $T$-module.
In Section 8, we introduce the matrix $A$ and describe how $A, A^*$ are related.
In Section 9, we show that $A$ is diagonalizable, with $2N+1$ eigenspaces. For each eigenspace, we describe the dimension  and eigenvalue.
In Section 10, we show that $A^*$ acts on the eigenspaces of $A$ in a (block) tridiagonal fashion. We use this action to show that $A$ is $Q$-polynomial.
In Section 11, we show that $A, A^*$ act on each irreducible $T$-module as a Leonard pair.

 \section{Preliminaries}
 
\noindent We now begin our formal argument. In this section, we review some basic concepts and establish some notation.
Recall the natural numbers $\mathbb N = \lbrace 0,1,2,\ldots \rbrace$ and integers $\mathbb Z = \lbrace 0, \pm 1, \pm 2, \ldots\rbrace$.
Let $\mathbb C$ denote the field of complex numbers. 
Let $X$ denote a nonempty  finite  set. An element of $X$ is called a {\it vertex}.
Let ${\rm Mat}_X(\mathbb C)$
denote the $\mathbb C$-algebra
consisting of the matrices with rows and columns  indexed by $X$
and all entries in $\mathbb C  $. Let $I \in{ \rm Mat}_X(\mathbb C)$ denote the identity matrix.
 Let
$V=\mathbb C^X$ denote the vector space over $\mathbb C$
consisting of the column vectors with
coordinates indexed by $X$ and all entries 
in $\mathbb C$.
The algebra ${\rm Mat}_X(\mathbb C)$
acts on $V$ by left multiplication.
We call $V$ the {\it standard module}.
We endow $V$ with a Hermitean form $\langle \, , \, \rangle$ 
that satisfies
$\langle u,v \rangle = u^t {\overline v}$ for 
$u,v \in V$,
where $t$ denotes transpose and $-$ denotes complex conjugation. 
For all $y \in X,$ define a vector $\hat{y} \in V$ that has $y$-coordinate  $1$ and all other coordinates $0$.
The vectors $\lbrace \hat y \rbrace_{y \in X}$ form an orthonormal basis for $V$.
\medskip

\noindent
Throughout the paper, we apply linear algebra using
the conventions and notation of \cite[pp.~5,~6]{nomSpinModel}.
 %
   
\section{The Attenuated Space poset $\mathcal A_q(N,M)$}

 In this section, we recall the Attenuated Space poset $\mathcal A_q(N,M)$. We refer the reader to \cite[Section~9.5]{bcn} and \cite{WenLiu, uniform, quantumMatroid} for background information about this poset.
 \medskip
 
 \noindent Let $\mathbb F_q$ denote a finite field with $q$ elements. Let $N, M$ denote positive integers. Let $H$ denote a vector space over $\mathbb F_q$
 that has dimension $N+M$. Fix a subspace $h \subseteq H$ that has dimension $M$. Let the set $X$ consist of the subspaces of $H$ that have zero intersection 
 with $h$.  Define a partial order $\leq $ on $X$ such that $y \leq z$ whenever $y \subseteq z$ $(y, z \in X)$. The partially ordered set
  $X, \leq$ is denoted by
 $\mathcal A_q(N,M)$ and
 called the {\it Attenuated Space poset}.  The poset $\mathcal A_q(N,M)$ is ranked with rank $N$; the rank of a vertex is equal to its dimension.
 \medskip
 
 \noindent As we discuss $\mathcal  A_q(N,M)$ the following concepts and notation will be useful. Let $y, z \in X$. Define $y < z$ whenever $y \leq z$ and $ y \not=z$. 
 We say that $z$ {\it covers} $y$ whenever $y < z$ and there does not exist $w \in X$ such that $y < w < z$. Note that $z$ covers $y$ if and only if $y \leq z$ and
 ${\rm dim}\,z-{\rm dim}\,y = 1$. We say that $y,z$ are {\it adjacent} whenever one of $y, z$ covers the other one.
  For $n \in \mathbb N$ define
  \begin{align*}
  \lbrack n \rbrack_q = \frac{q^n-1}{q-1}.
  \end{align*}
  \noindent We further define
  \begin{align*}
  \lbrack n \rbrack^!_q =  \lbrack n \rbrack_q  \lbrack n-1 \rbrack_q\cdots \lbrack 2 \rbrack_q \lbrack 1 \rbrack_q.
  \end{align*}
  \noindent We interpret $\lbrack 0 \rbrack^!_q =1$. For  $0 \leq i \leq n$ define
  \begin{align*}
  \binom{n}{i}_q =  \frac{\lbrack n \rbrack^!_q}{\lbrack i \rbrack^!_q \lbrack n-i\rbrack^!_q}.
  \end{align*}
  For notational convenience, define $\binom{n}{i}_q=0$ for all integers $i$ such that $i<0$ or $i>n$.
\medskip

\noindent The following results are well known; see for example \cite{quantumMatroid}.
 
 \begin{lemma} Let $0 \leq i \leq N$ and let $y \in X$ have dimension $i$.
 \begin{enumerate}
 \item[\rm (i)] $y$ covers exactly $\lbrack i \rbrack_q $ vertices;
 \item[\rm (ii)] $y$ is covered by exactly $q^M \lbrack N-i\rbrack_q$ vertices.
 \end{enumerate}
 \end{lemma}
 \begin{lemma} \label{lem:subSize} For $0 \leq i \leq N$, the number of vertices in $X$ that have dimension $i$ is equal to $q^{Mi} \binom{N}{i}_q$.
 \end{lemma} 
 
 \section{The matrix $A^*$ }
 We continue to discuss the Attenuated Space poset $\mathcal A_q(N,M)$. 
 We are going to define two matrices in ${\rm Mat}_X(\mathbb C)$, denoted by $A$ and $A^*$.
 We will discuss $A^*$ in this section. We will discuss $A$ in  Section 8.
 \medskip
 
\begin{definition} \label{def:Es} \rm For $0 \leq i \leq N$ define a diagonal matrix $E^*_i \in {\rm Mat}_X(\mathbb C)$ that has $(y,y)$-entry
 \begin{align*}
(E^*_i)_{y,y} = \begin{cases} 1, &{\mbox{\rm if ${\rm dim}\,y = i$}}; \\
                                         0, & {\mbox{\rm if ${\rm dim}\,y \not=i$}}
                \end{cases}  \qquad \qquad (y \in X).
 \end{align*}
 \end{definition}
 
\noindent
We have some comments about Definition \ref{def:Es}.
\begin{lemma} \label{def:Es1} For $0 \leq i \leq N$ and $y \in X$,

 \begin{align*}
E^*_i \hat y= \begin{cases} \hat y, &{\mbox{\rm if ${\rm dim}\,y = i$}}; \\
                                         0, & {\mbox{\rm if ${\rm dim}\,y \not=i$}}.
                \end{cases}
 \end{align*}
 \end{lemma}
 \begin{proof} By Definition \ref{def:Es}.
 \end{proof}

\noindent Recall the standard module $V$.

\begin{lemma} \label{lem:c1}
 For $0 \leq i \leq N$ we have
 \begin{align*}
 E^*_iV = {\rm Span} \lbrace \hat y \vert y \in X, \;\; {\rm dim}\, y=i\rbrace.
 \end{align*}
 Moreover, 
 \begin{align*}
 V=\sum_{i=0}^N E^*_iV \qquad  \mbox{\rm (orthogonal direct sum)}.
 \end{align*}
 \end{lemma}
 \begin{proof} By Lemma \ref{def:Es1} and since $\lbrace {\hat y} \rbrace_{y \in X}$ is an orthonormal basis for $V$.
 \end{proof}
 
 \begin{lemma} We have 
 \begin{align*}
 {\rm dim}\, E^*_iV = q^{Mi} \binom{N}{i}_q \qquad \qquad (0 \leq i \leq N).
 \end{align*}
 \end{lemma}
 \begin{proof} By Lemma \ref{lem:subSize} and the first assertion in Lemma \ref{lem:c1}.
 \end{proof}

 \begin{definition}\rm For $0 \leq i \leq N$, we call $E^*_iV$ the {\it $i$th subconstituent} of $\mathcal A_q(N,M)$.
 \end{definition}

 \begin{lemma} \label{lem:basicEs} We have  $E^*_i E^*_j = \delta_{i,j} E^*_i$ for $0 \leq i,j\leq N$. Moreover,
 $I = \sum_{i=0}^N E^*_i$.
\end{lemma}
\begin{proof} By Definition  \ref{def:Es}.
\end{proof}

\noindent  By Lemma \ref{lem:basicEs}, the matrices $\lbrace E^*_i \rbrace_{i=0}^N$ form a basis for 
 a commutative subalgebra of ${\rm Mat}_X(\mathbb C)$; we denote this subalgebra by $\bf M^*$.
\medskip

 \begin{definition} \label{def:As} \rm Define a diagonal matrix $A^* \in {\rm Mat}_X(\mathbb C)$ that has $(y,y)$-entry $q^{-{\rm dim} \,y}$ for $y \in X$. 
 \end{definition}
 \noindent Note that $A^*$ is invertible. 
 \begin{lemma} \label{lem:Ky} For $y \in X$,
\begin{align*}
A^* \hat y = q^{-{\rm dim}\,y} \hat y.
\end{align*}
\end{lemma}
\begin{proof} By Definition  \ref{def:As}.
\end{proof}

\noindent By construction, the eigenvalues of $A^*$ are $\lbrace q^{-i}\rbrace_{i=0}^N$.
\begin{lemma} \label{lem:AsES}
 For $0 \leq i \leq N$, 
 $E^*_iV$ is the eigenspace of $A^*$ for the eigenvalue $q^{-i}$.
\end{lemma}
\begin{proof}
By Lemmas  \ref{lem:c1}, \ref{lem:Ky}.
\end{proof}

\begin{lemma} \label{lem:AsE}
 We have
 \begin{align*}
 A^* = \sum_{i=0}^N q^{-i} E^*_i.
 \end{align*}
Moreover, the subalgebra $\bf M^*$ is generated by $A^*$.
\end{lemma}
\begin{proof} The first assertion follows from Definitions \ref{def:Es}, \ref{def:As}. 
 The second assertion holds because
$\lbrace q^{-i} \rbrace_{i=0}^N$ are mutually distinct.
\end{proof}

\begin{remark}\rm The matrix $A^*$ is a variation on the matrix $K$ that appears in \cite[Section~2]{WenLiu}.  By construction $K = q^{N+M} A^*$.
\end{remark} 

\begin{definition}\label{def:Smat} \rm Define a matrix
\begin{align} \label{eq:Sdef}
S = \sum_{i=0}^N (-1)^i E^*_i.
\end{align}
\end{definition}
\begin{lemma} \label{lem:Stwo} We have 
\begin{enumerate}
\item[\rm (i)] $S^2 = I$;
\item[\rm (ii)]  $S A^* = A^* S$.
\end{enumerate}
\end{lemma}
\begin{proof} By  Lemmas \ref{lem:basicEs}, \ref{lem:AsE} and Definition \ref{def:Smat}.
\end{proof}
 
 \section{Some raising and lowering matrices}
\noindent  
 We continue to discuss the Attenuated Space poset $\mathcal A_q(N,M)$. 
  We are going to define a raising matrix $R$ and two lowering matrices $L, L'$. We comment on how $L, L'$ differ.
  Roughly speaking,
  $L$ is more natural from an algebraic point of view, and $L'$ is more natural from a combinatorial point of view.
 In this section we will describe $R$ and $L'$. We will describe $L$ in Section 7.
 
 \begin{definition}\label{def:RL} \rm We define matrices $R, L'$ in ${\rm Mat}_X(\mathbb C)$ that have $(y,z)$-entries
\begin{align*}
R_{y,z} = \begin{cases} 1, &{\mbox{\rm if $y$  covers $z$}}; \\
                                         0, & {\mbox{\rm if $y$ does not cover $z$}}
                \end{cases} 
\qquad   \quad               
L'_{y,z} = \begin{cases} 1, &{\mbox{\rm if $z$  covers $y$}}; \\
                                         0, & {\mbox{\rm if $z$ does not cover $y$}}
                \end{cases}                
\end{align*}
for $y,z \in X$. Note that $L'=R^t$, where $t$ denotes transpose. We call $R$ (resp. $L'$) the {\it raising matrix} (resp. {\it lowering matrix}) of $\mathcal A_q(N,M)$.
\end{definition}

\begin{lemma} \label{lem:RLy} For $z \in X$ we have
\begin{align*}
R \hat z = \sum_{y \;{\rm covers}\; z} \hat y, \qquad \qquad L' \hat z=  \sum_{z \;{\rm covers}\; y} \hat y.
\end{align*}
\end{lemma}
\begin{proof} By Definition \ref{def:RL}.
\end{proof}
\noindent Next, we describe how $R, L'$ are related to $\lbrace E^*_i \rbrace_{i=0}^N$.
\begin{lemma} \label{lem:Rexpand} We have
\begin{align*}
&R E^*_i = E^*_{i+1} R \quad (0 \leq i \leq N-1), \qquad \quad  E^*_0 R=0, \qquad \quad R E^*_N=0;
\\
&L' E^*_i = E^*_{i-1} L' \quad (1 \leq i \leq N), \qquad \quad  L' E^*_0 =0, \qquad \quad E^*_N L'=0.
\end{align*}
\end{lemma}
\begin{proof} To verify these equations, for $z \in X$ apply each side to $\hat z$, and evaluate the
results using Lemmas  \ref{def:Es1},  \ref{lem:RLy}.
\end{proof} 

\begin{lemma} \label{lem:RLaction} We have
  \begin{align*}
  & R E^*_i V \subseteq E^*_{i+1} V \qquad \quad (0 \leq i \leq N-1), \qquad \quad R E^*_NV=0;\\
& L' E^*_i V \subseteq E^*_{i-1}V \qquad (1 \leq i \leq N), \qquad \quad L' E^*_0V =0.
 \end{align*}
\end{lemma}
\begin{proof} By Lemma  \ref{lem:RLy} or  \ref{lem:Rexpand}.
\end{proof}

\begin{lemma} We have
\begin{align*}
R^{N+1}=0, \qquad \qquad (L')^{N+1}=0.
\end{align*}
\end{lemma} 
\begin{proof} By Lemma \ref{lem:RLaction}.
\end{proof}

 \begin{lemma} \label{lem:SRL} We have
 \begin{align*}
 S R = -R S, \qquad \qquad S L' = - L' S.
 \end{align*}
 \end{lemma}
 \begin{proof} To verify these relations, eliminate $S$ using \eqref{eq:Sdef}, and
 evaluate the results using
 Lemma  \ref{lem:Rexpand}.
 \end{proof}

 \begin{lemma} \label{lem:RLK} {\rm (See \cite[Lemma~3.1]{WenLiu}).} The matrices $R,L', A^*$ satisfy the following relations:
 \begin{align}
&  RA^*=q A^* R, \qquad \qquad L' A^*=q^{-1}  A^* L', \label{eq:RLK1} \\
 &(L')^2 R-(q+1) L'RL' + q R(L')^2 = - (q+1)q^{N+M}L' A^*, \label{eq:RLK2} \\
& L' R^2 - (q+1) RL'R + q R^2 L' = - (q+1) q^{N+M} A^*R. \label{eq:RLK3} 
 \end{align}
 \end{lemma}

\section{The subconstituent algebra $T$ and its irreducible modules}
 \noindent  We continue to discuss the Attenuated Space poset $\mathcal A_q(N,M)$. 
 \begin{definition}\label{def:SubT} \rm Let $T$ denote the subalgebra of ${\rm Mat}_X(\mathbb C)$ generated by $R,L', A^*$.
 We call $T$ the {\it subconstituent algebra} of $\mathcal A_q(N,M)$.
 \end{definition}
 
 \begin{remark} \rm In \cite{uniform} the algebra $T$ is called the incidence algebra. 
 \end{remark}
 
 \noindent We refer the reader to \cite{caughman1, cerzo, go, terwSub1, terwSub2, terwSub3, int, zitnik} for background information about the algebra $T$.
 \medskip
 
 \noindent Recall the standard module $V$. By a {\it $T$-module} we mean a subspace $ W \subseteq V$ such that $T W \subseteq W$.
 A $T$-module $W$ is said to be {\it irreducible} whenever $W\not=0$ and $W$ does not contain a $T$-module besides $0$ and $W$.
 \medskip
 
 \noindent  The following comments about $T$-modules are adapted from  \cite[Sections~2, 3]{zitnik}.
 The algebra $T$ is closed under the conjugate-transpose map, because the generators $R, L', A^*$ have all entries real and $R^t=L'$, $(A^*)^t = A^*$.
 The above closure implies that
  the orthogonal complement of a $T$-module is a $T$-module. Consequently, every $T$-module is an orthogonal direct sum of
 irreducible $T$-modules. In particular, the standard module
  $V$ is an orthogonal direct sum of irreducible $T$-modules. The irreducible $T$-modules are
 described as follows. Let $W$ denote an irreducible $T$-module. Then $W$ is an orthogonal direct sum of the nonzero subspaces among $\lbrace E^*_iW\rbrace_{i=0}^N$.
By the {\it endpoint} of $W$ we mean ${\rm min} \lbrace i \vert 0 \leq i \leq N, \; E^*_iW \not=0\rbrace$.
By the {\it diameter} of $W$ we mean $\bigl \vert \lbrace i \vert 0 \leq i \leq N, \; E^*_i W \not=0\rbrace \bigr \vert -1$.

 \begin{lemma} \label{lem:Iso1} {\rm (See \cite[Theorems~2.5,~3.3]{uniform}).} For $0 \leq r,d\leq N$ the following are equivalent:
 \begin{enumerate}
 \item[\rm (i)] there exists an irreducible $T$-module with endpoint $r$ and diameter $d$;
 \item[\rm (ii)] $ N-2r \leq d \leq N-r$ and $d \leq N+M-2r$.
 \end{enumerate}
 \end{lemma}
 \noindent 
 Let $W$ and $W'$ denote $T$-modules. By an {\it isomorphism of $T$-modules from $W$ to $W'$}, we mean an isomorphism of vector 
 spaces $\zeta : W \to W'$ such that $ \zeta  \eta= \eta \zeta $ for all $\eta \in T$. The $T$-modules $W, W'$ are said to be {\it isomorphic}
 whenever there exists an isomorphism of $T$-modules from $W$ to $W'$. 
 
 \begin{lemma}\label{lem:Iso2} {\rm (See \cite[Theorem~2.5]{uniform}).}
 Let $W$ and $W'$ denote irreducible $T$-modules, with endpoints $r, r'$  and diameters $d, d'$ respectively. Then the following are equivalent:
 \begin{enumerate}
 \item[\rm (i)]  $W$ and $W'$ are isomorphic;
 \item[\rm (ii)] $r=r'$ and $d=d'$.
 \end{enumerate}
 \end{lemma}
 
 \noindent Following \cite[Section~5]{WenLiu}  let $\Psi$ denote the set of isomorphism classes of irreducible $T$-modules. By Lemmas \ref{lem:Iso1}, \ref{lem:Iso2} we view
 \begin{align} \label{eq:PsiRD}
 \Psi= \lbrace (r,d) \;\vert \; 0 \leq r,d\leq N, \quad N - 2r \leq d \leq N-r, \quad d \leq N+M-2r\rbrace.
 \end{align} 
 
 \begin{definition} \label{def:xip} \rm {\rm (See \cite[line (16)] {Terint}).} For $(r,d) \in \Psi$ define
 \begin{align*}
 \xi'_{i} (r,d) = \frac{ q^{N+M-r-d} (q^i-1)(q^{d+1-i}-1)}{(q-1)^2} \qquad \qquad (1 \leq i \leq d).
 \end{align*}
 \end{definition}
 \noindent Referring to Definition \ref{def:xip}, we have $\xi'_{i}(r,d)\not=0$ for $1 \leq i \leq d$.

 \begin{proposition} \label{lem:wbasis}  {\rm (See \cite[p.~78]{Terint}).} Let $W$ denote an irreducible $T$-module, with endpoint $r$ and diameter $d$.
 There exists a basis $\lbrace w_i \rbrace_{i=0}^d$ of $W$ such that
 \begin{enumerate}
 \item[\rm (i)] $w_i \in E^*_{r+i} V\quad (0 \leq i \leq d)$;
 \item[\rm (ii)] $Rw_i = w_{i+1} \quad  (0 \leq i \leq d-1), \qquad  Rw_d =0$;
 \item[\rm (iii)] $L'w_i = \xi'_{i} (r,d) w_{i-1} \quad (1 \leq i \leq d), \qquad  L' w_0=0$.
 \end{enumerate}
 \end{proposition}
 
 \begin{lemma} \label{lem:AsW} Referring to Proposition  \ref{lem:wbasis},
 we have $A^* w_i = q^{-r-i} w_i$ for $0 \leq i \leq d$.
 \end{lemma} 
 \begin{proof} By Lemma \ref{lem:AsE} and Proposition \ref{lem:wbasis}(i).
 \end{proof}
 
 \begin{definition}\label{def:mult} \rm For $(r,d) \in \Psi$ let ${\rm mult}(r,d)$ denote the multiplicity with which the irreducible
 $T$-module with endpoint $r$ and diameter $d$ appears in the standard module $V$.
 \end{definition}
 \noindent Our next goal is to express ${\rm mult}(r,d)$ in terms of $r$ and $d$.
 \medskip
  
 \noindent We bring in some notation. For an integer $r$ define
\begin{align}
\mu_r = \binom{N}{r}_q \binom{M}{r}_q (q^r-1)(q^r-q) \cdots (q^r-q^{r-1}). \label{eq:muForm}
\end{align}
\noindent Note that $\mu_r\not=0$ if and only if $0 \leq r\leq {\rm min}(N,M)$.
\medskip

 \noindent The following is a reformulation of item 7 in \cite[Theorem~3.3]{uniform}.
 \begin{proposition} \label{prop:multForm} {\rm (See \cite[Theorem~3.3]{uniform}).} Let $(r,d) \in \Psi$. 
 \begin{enumerate}
 \item[\rm (i)] If $r+d=N$ then ${\rm mult}(r,d)=\mu_r$.
 \item[\rm (ii)] If $r+d\leq N-1$ then
 \begin{align*}
 {\rm mult}(r,d) = \mu_{2r+d-N} \Biggl( \binom{2N-2r-d}{N-r-d}_q - \binom{2N-2r-d}{N-r-d-1}_q \Biggr).
 \end{align*}
 \end{enumerate}
 \end{proposition}
 
 \begin{proposition} \label{prop:PsiSum} Let $(r,d) \in \Psi$. Then $\Psi$ contains
 \begin{align*}
 (r-\ell, d+2\ell) \qquad \qquad (0 \leq \ell \leq N-r-d).
 \end{align*}
 \noindent Moreover
 \begin{align} \label{eq:sumRD}
  \sum_{\ell=0}^{N-r-d} {\rm mult}(r-\ell, d+2\ell)=
 \mu_{2r+d-N} \binom{2N-2r-d}{N-r-d}_q.
 \end{align}
 \end{proposition} 
 \begin{proof} The first assertion follows from \eqref{eq:PsiRD}. To verify \eqref{eq:sumRD}, evaluate the left-hand side
 using Proposition \ref{prop:multForm} and simplify the result.
 \end{proof}

 \section{The $q$-lowering matrix  $L$}

 \noindent  We continue to discuss the Attenuated Space poset $\mathcal A_q(N,M)$.  Recall the
 lowering matrix $L'$ from Definition \ref{def:RL}.
 In this section, we adjust $L'$ to get a matrix $L$ called the $q$-lowering matrix. We describe how $L$ is related to the matrices $A^*, R$ from Definitions \ref{def:As}, \ref{def:RL} respectively.
 We also describe how $L$ acts on each irreducible $T$-module.
  \medskip
  
  \noindent The following definition is motivated by \cite[Section~1]{murali}.
 \begin{definition}\label{def:RLbf} \rm We define a matrix $ L \in {\rm Mat}_X(\mathbb C)$ that has $(y,z)$-entry
\begin{align*}
L_{y,z} = \begin{cases} q^{{\rm dim}\,y}, &{\mbox{\rm if $z$  covers $y$}}; \\
                                         0, & {\mbox{\rm if $z$ does not cover $y$}}
                \end{cases}                \qquad \qquad (y, z \in X).
\end{align*}
We call $L$ the {\it $q$-lowering matrix} for $\mathcal A_q(N,M)$.
\end{definition}

\begin{lemma} \label{lem:RLybf} For $z \in X$ we have
\begin{align*}
 L \hat z=  \sum_{z \;{\rm covers}\; y} q^{{\rm dim}\,y} \hat y.
\end{align*}
\end{lemma}
\begin{proof} By Definition \ref{def:RLbf}.
\end{proof}

\begin{lemma}\label{lem:LL} We have  $L'=A^* L$.
\end{lemma}
\begin{proof} Compare Lemma \ref{lem:RLybf} with Lemmas \ref{lem:Ky}, \ref{lem:RLy}.
\end{proof}

\begin{lemma} \label{lem:Lexpand} We have
\begin{align*}
&L E^*_i = E^*_{i-1} L \quad (1 \leq i \leq N), \qquad \quad  LE^*_0 =0, \qquad \quad E^*_N L=0.
\end{align*}
\end{lemma}
\begin{proof} Use Lemmas \ref{lem:Rexpand}, \ref{lem:LL}.
\end{proof}

\begin{lemma} \label{lem:Laction2}  We have $LE^*_iV \subseteq E^*_{i-1} V $ for $1 \leq i \leq N$, and $L E^*_0V=0$.
\end{lemma}
\begin{proof} By Lemma  \ref{lem:RLybf} or \ref{lem:Lexpand}.
\end{proof}

\begin{lemma} We have $L^{N+1}=0$.
\end{lemma}
\begin{proof} By Lemma \ref{lem:Laction2}.
\end{proof}

 \begin{lemma} \label{lem:SL} We have $SL=-LS$.
 \end{lemma}
 \begin{proof} By Lemmas  \ref{lem:Stwo}(ii),  \ref{lem:SRL}, \ref{lem:LL}.
  \end{proof}

 \begin{proposition} \label{lem:RLAs} The matrices $R,L, A^*$ satisfy the following relations:
 \begin{align*}
 &RA^*=q A^* R, \qquad \qquad L A^*=q^{-1}  A^* L, \\
  & L^2 R - q(q+1) LRL + q^3 R L^2 = - q^{N+M}(q+1) L,\\
 & LR^2 - q(q+1) RLR + q^3 R^2 L = - q^{N+M} (q+1)R.
 \end{align*}
 \end{proposition}
 \begin{proof} In Lemma \ref{lem:RLK}, eliminate $L'$ using Lemma \ref{lem:LL}.
 \end{proof}
 
 \begin{remark}\rm The last two relations in Proposition \ref{lem:RLAs} 
 are a special case of the down-up relations \cite[p.~308]{benkart}.
 \end{remark}
 
 \begin{proposition} \label{prop:TRLAs} The algebra $T$ is generated by $R,L, A^*$.
 \end{proposition}
 \begin{proof} By Definition \ref{def:SubT}, Lemma \ref{lem:LL}, and since $A^*$ is invertible.
 \end{proof}

 \noindent Next, we describe how $L$ acts on each irreducible $T$-module. In this description, we will use the following variation on Definition \ref{def:xip}.
 
 \begin{definition} \label{def:xi} \rm For $(r,d) \in \Psi$ define
 \begin{align*}
 \xi_{i} (r,d) = \frac{ q^{N+M-d} (q^i-1)(q^d-q^{i-1})}{(q-1)^2} \qquad \qquad (1 \leq i \leq d).
 \end{align*}
 \end{definition}
 \noindent Comparing  Definitions \ref{def:xip}, \ref{def:xi} we obtain $\xi_i(r,d) = q^{r+i-1} \xi'_i(r,d)$ for $1 \leq i \leq d$.

 \begin{proposition} \label{prop:LW} Let $W$ denote an irreducible $T$-module, with endpoint $r$ and diameter $d$.
 Recall the basis $\lbrace w_i \rbrace_{i=0}^d$ of $W$ from Proposition \ref{lem:wbasis}. Then
 \begin{align*}
 Lw_i = \xi_{i} (r,d) w_{i-1} \qquad (1 \leq i \leq d), \qquad \quad L w_0=0.
 \end{align*}
 \end{proposition}
 \begin{proof}  Eliminate $L'$ in Proposition \ref{lem:wbasis}(iii) using $L'=A^* L$, and evaluate the result using Lemma \ref{lem:AsES} along with
  $\xi_i(r,d) = q^{r+i-1} \xi'_i(r,d)$ for $1 \leq i \leq d$.
 \end{proof}

 \section{The $q$-adjacency matrix $A$ and its relationship to $A^*$}
  \noindent  We continue to discuss the Attenuated Space poset $\mathcal A_q(N,M)$.  In this section, we  introduce the $q$-adjacency matrix $A$ and describe
  its relationship to $A^*$.
 \medskip
 
  \noindent The following definition is motivated by \cite[Section~1]{murali}.

 \begin{definition}\label{def:A} \rm Define a matrix $A \in {\rm Mat}_X(\mathbb C)$ that has $(y,z)$-entry
\begin{align*}
A_{y,z} = \begin{cases} 1 &{\mbox{\rm if $y$  covers $z$}}; \\
                                      q^{{\rm dim}\,y} &{\mbox{\rm if $z$ covers $y$}}; \\
                                         0 & {\mbox{\rm if $y,z$ are not adjacent}}
                \end{cases} 
                \qquad \qquad (y, z \in X).
\end{align*} We call $A$ the {\it $q$-adjacency matrix} for $\mathcal A_q(N,M)$.
\end{definition}

\noindent Note that $A$ is a weighted adjacency matrix for $\mathcal A_q(N,M)$, in the sense of \cite[Definition~2.1]{Lnq}. Next, we clarify how $A$ is related to $L$ and $R$.

\begin{lemma} \label{lem:Ay} We have $A=R+L$. Moreover
\begin{align} \label{eq:RLexpand} 
R = \sum_{i=0}^{N-1} E^*_{i+1} A E^*_i, \qquad \qquad L = \sum_{i=1}^N  E^*_{i-1} A E^*_i.
\end{align}
\end{lemma}
\begin{proof} The first assertion is from Definition \ref{def:A}. To verify \eqref{eq:RLexpand}, eliminate $A$ using $A=R+L$ and evaluate the
results using Lemmas \ref{lem:basicEs}, \ref{lem:Rexpand}, \ref{lem:Lexpand}.
\end{proof}

\begin{lemma} \label{lem:AsTD} We have
\begin{align*}
&A E^*_iV \subseteq E^*_{i-1} V + E^*_{i+1}V \qquad (1 \leq i \leq N-1), \\
& A E^*_0V \subseteq E^*_1V, \qquad \qquad A E^*_NV \subseteq E^*_{N-1}V.
\end{align*}
\end{lemma}
\begin{proof} By $A=R+L$ and Lemmas  \ref{lem:RLaction},  \ref{lem:Laction2}.
\end{proof}

\begin{lemma} \label{lem:SAAS}  We have $SA=-AS$.
\end{lemma}
\begin{proof} We have $A=R+L$ and $SR=-RS$ and $SL=-LS$.
\end{proof}

\begin{lemma} The subconstituent algebra $T$ is generated by $A, A^*$.
\end{lemma} 
\begin{proof} By Proposition \ref{prop:TRLAs} and Lemma \ref{lem:Ay}, along with the fact that $E^*_i$ is a polynomial in $A^*$ for $0 \leq i \leq N$.
\end{proof}

\noindent Next, we describe how $A, A^*$ are related. For notational convenience, define
\begin{align}
\beta = q + q^{-1}.
\label{eq:beta}
\end{align}

\begin{proposition} \label{prop:TD1} 
The matrices $A$ and $A^*$ satisfy
\begin{align*}
&A^3 A^* - (\beta+ 1) A^2 A^* A + (\beta+1) A A^* A^2 - A^* A^3\\
& \qquad \qquad \qquad  = q^{N+M-2} (q+1)^2 (A A^*-A^*A), \\
& (A^*)^2 A - \beta A^* A A^*  + A (A^*)^2 = 0.
\end{align*}
\end{proposition} 
\begin{proof}  To verify these relations, eliminate $A$ using $A=R+L$, and evaluate the
results using Proposition  \ref{lem:RLAs}. Here are some details for the first relation.
Define
\begin{align*}
\Delta = &A^3 A^* - (\beta + 1) A^2 A^* A + (\beta+1) A A^* A^2 - A^* A^3\\
& \qquad \qquad \qquad  -q^{N+M-2} (q+1)^2 (A A^*-A^*A).
\end{align*}
\noindent We show that $\Delta=0$. By Proposition  \ref{lem:RLAs} we  have  $A^*R=q^{-1} R A^*$ and $A^* L = q L A^*$. Observe that
\begin{align*}
A^* A = A^* (R+L) = (q^{-1} R + q L) A^*.
\end{align*}
\noindent We have
\begin{align*}
&A^3 A^* = (R+L)^3 A^*, \qquad \qquad A^2 A^* A = (R+L)^2 (q^{-1} R + q L)A^*, \\
& A A^* A^2 = (R+L)(q^{-1} R  + q L)^2 A^*, \qquad \qquad A^* A^3 = (q^{-1} R + q L)^3 A^*, \\
& A A^*-A^* A = (q-1) (q^{-1} R-L) A^*.
\end{align*}
We evaluate $\Delta$ using these equations. After some algebraic manipulation, we obtain
\begin{align*}
\Delta &= \biggl ( L^2 R - q (q+1) LRL + q^3 R L^2 + q^{N+M} (q+1) L \biggr) A^* (1-q^{-2}) \\
            & \; + \biggl( L R^2 - q(q+1) RLR + q^3 R^2 L + q^{N+M} (q+1) R \biggr) A^* q^{-1} (q^{-2}-1).
\end{align*}
\noindent In this equation, the expression inside each large parenthesis is zero by Proposition  \ref{lem:RLAs}. Therefore $\Delta=0$.
 \end{proof}

\noindent The second equation in Proposition \ref{prop:TD1} has the following consequence.
\begin{proposition} \label{cor:TD2}
We have
\begin{align*}
(A^*)^3 A - (\beta + 1) (A^*)^2 A A^* + (\beta+1) A^* A (A^*)^2 - A (A^*)^3 = 0.
\end{align*}
\end{proposition} 
\begin{proof} Compute the commutator of $A^*$ with the second equation in Proposition \ref{prop:TD1}.
\end{proof}
\begin{remark}\rm The equation in Proposition \ref{cor:TD2} and the first equation in Proposition \ref{prop:TD1},
give a special case of the tridiagonal relations \cite[lines (14), (15)]{qSerre}.
\end{remark}

\section{The eigenvalues of $A$}
 \noindent  We continue to discuss the Attenuated Space poset $\mathcal A_q(N,M)$.  In this section, we compute the
 eigenvalues of the $q$-adjacency matrix $A$. Our strategy is to compute the eigenvalues for the action of $A$ on each irreducible $T$-module.

\begin{lemma} \label{prop:AonW} Let $W$ denote an irreducible $T$-module, with endpoint $r$ and diameter $d$.
 Consider the matrix that represents $A$ with respect to the basis $\lbrace w_i \rbrace_{i=0}^d$ of $W$ from Proposition  \ref{lem:wbasis}.
This matrix is tridiagonal with entries
\begin{align}
\left(
\begin{array}{ c c cc c c }
 0 & \xi_1   &   &&   & \bf 0  \\
 1 & 0  & \xi_2 &&  &      \\ 
   & 1  &  0   & \cdot &&  \\
     &   & \cdot & \cdot  & \cdot & \\
       &  & & \cdot & \cdot & \xi_d \\
        {\bf 0}  &&  & & 1  &  0  \\
	\end{array}
	\right),
\label{eq:Deltamat}
\end{align}
\noindent where the scalars $\xi_i = \xi_i(r,d)$ are from Definition \ref{def:xi}.
\end{lemma}
\begin{proof} By $A=R+L$ and 
 Propositions  \ref{lem:wbasis}, \ref{prop:LW}.
 \end{proof}
 
 \noindent Shortly we will compute the eigenvalues of the matrix \eqref{eq:Deltamat}.
 \medskip

 \noindent We bring in some notation.
\begin{definition} \label{def:thi} \rm Define the set
\begin{align*}
\lbrack N \rbrack & = \lbrace i \;\vert \; 2 i \in \mathbb Z, \quad 0 \leq i \leq N\rbrace
\\
&= \lbrace 0, 1/2, 1, 3/2, \ldots,N\rbrace.
\end{align*}
\end{definition}
\noindent The cardinality of $\lbrack N \rbrack$ is $2N+1$.
\medskip

\noindent 
The scalar $q$ is real and $q\geq 2$;  let $q^{1/2}$ denote the positive square root of $q$. 
\begin{definition}\label{def:th} For $i \in \lbrack N \rbrack$ define
\begin{align*}
\theta_i = \frac{ q^{N-i}-q^i}{q-1}\,q^{M/2}.
\end{align*}
\end{definition}

\begin{lemma} \label{lem:3step} The following {\rm (i)--(iii)} hold:
\begin{enumerate}
\item[\rm (i)]
for $i,j \in \lbrack N \rbrack$ we have
\begin{align} \label{eq:diff}
\theta_i - \theta_j = \frac{(q^j-q^i)(q^{N-i-j} + 1) q^{M/2}}{q-1};
\end{align}
\item[\rm (ii)] the scalars $\lbrace \theta_i \rbrace_{i \in \lbrack N \rbrack}$ are mutually distinct;
\item[\rm (iii)]  $\theta_{N-i} = - \theta_i $ for $i \in \lbrack N \rbrack$.
\end{enumerate}
\end{lemma}
\begin{proof} (i) By Definition \ref{def:th}. \\
\noindent (ii) By (i) and since $q$ is not a root of unity. \\
\noindent (iii) By Definition \ref{def:th}.
\end{proof}
\noindent Let $W$ denote an irreducible $T$-module. Let $d$ denote the diameter of $W$, and define $t=(N-d)/2$. 
We will need the fact that $t + i \in \lbrack N \rbrack$ for $0 \leq i \leq d$. This is the case, 
 because 
 \begin{align*}
 2(t+i) =N-d+2i \in \mathbb Z
 \end{align*}
 and
\begin{align*}
0 \leq t+i \leq t+d = \frac{N-d}{2} + d = \frac{N+d}{2} \leq \frac{N+N}{2} = N.
\end{align*}

\noindent Consider a  linear map on a finite-dimensional vector space. This map is called {\it multiplicity-free} whenever the map is diagonalizable, and
each eigenspace has dimension one.

\begin{lemma} \label{lem:eigvalW} Let $W$ denote an irreducible $T$-module. Let $d$ denote the diameter of $W$, and define $t=(N-d)/2$. Then
the action of $A$ on $W$ is multiplicity-free, with eigenvalues $\lbrace \theta_{t+i} \rbrace_{i=0}^d$.
\end{lemma}
\begin{proof} We compute the eigenvalues of the matrix \eqref{eq:Deltamat}.
Define the tridiagonal matrix 
\begin{align*}
B =\left(
\begin{array}{ c c cc c c }
 0 & q^d-1   &   &&   & \bf 0  \\
 q-1 & 0  & q^d-q  &&  &      \\ 
   & q^2-1  &  0   & \cdot &&  \\
     &   & \cdot & \cdot  & \cdot & \\
       &  & & \cdot & \cdot & q^d-q^{d-1} \\
        {\bf 0}  &&  & & q^d-1  &  0  \\
	\end{array}
	\right).
\end{align*}
By \cite[Lemma~4.20]{lrt}, the matrix $B$ is multiplicity-free with eigenvalues
\begin{align*}
\label{eq:evListPre}
 q^{d-i} - q^i \qquad \qquad (0 \leq i \leq d).
\end{align*}
We now compare $B$ with the matrix \eqref{eq:Deltamat}.
Define a diagonal matrix $F = {\rm diag}\,(f_0, f_1, \ldots, f_d)$ such that $f_0=1$
and 
\begin{align*}
f_i = f_{i-1} \frac{q^i-1}{q-1} q^{M/2+t} \qquad \qquad (1 \leq i \leq d).
\end{align*}
By linear algebra, the matrix \eqref{eq:Deltamat} is equal to $\gamma F^{-1} B F$, where  $\gamma=(q-1)^{-1} q^{M/2+t}$. Therefore,
the matrix  \eqref{eq:Deltamat}  and the matrix $\gamma B$ have the same characteristic polynomial.
Consequently, the matrix \eqref{eq:Deltamat} is multiplicity-free with eigenvalues
\begin{align*}
\frac{ q^{d-i}-q^i}{q-1}\, q^{M/2+t} \qquad \qquad (0 \leq i \leq d).
\end{align*}
\noindent By Definition \ref{def:th},
\begin{align*}
\theta_{t+ i}=  \frac{ q^{d-i}-q^i}{q-1}\, q^{M/2+t} \qquad \qquad (0 \leq i \leq d).
\end{align*}
\noindent We have shown that the matrix \eqref{eq:Deltamat} is multiplicity-free with eigenvalues
$\lbrace \theta_{t+i} \rbrace_{i=0}^d$. The result follows.
\end{proof}

\begin{theorem} \label{thm:AeigVal} The matrix $A$ is diagonalizable, with eigenvalues  $\lbrace \theta_i \rbrace_{i\in \lbrack N \rbrack}$.
\end{theorem}
\begin{proof} The matrix $A$ is diagonalizable, because the standard module $V$ is a direct sum of irreducible $T$-modules, and $A$ is
diagonalizable on each irreducible $T$-module. 
 Let $\theta$ denote an eigenvalue of $A$. We show that $\theta$ is included among $\lbrace \theta_i \rbrace_{i\in \lbrack N \rbrack}$.
We mentioned that $V$ is a direct sum of irreducible $T$-modules.  So $\theta$ is 
an eigenvalue for the action of $A$ on some irreducible $T$-module $W$. Let  $d$ denote the diameter of $W$, and define $t=(N-d)/2$.
By Lemma \ref{lem:eigvalW}, there exists an integer $j$ $(0 \leq j \leq d)$ such that $\theta=\theta_{t+j}$. We have
$t+j \in \lbrack N \rbrack$ by the discussion above Lemma \ref{lem:eigvalW}. By these comments,
$\theta$ is included among $\lbrace \theta_i \rbrace_{i\in \lbrack N \rbrack}$. Conversely, for $i \in \lbrack N \rbrack$ we show that $\theta_i$ is an eigenvalue of $A$. 
 If $i \in \mathbb Z$ then $i$ is one of $0,1,2,\ldots, N$.
 Therefore $\theta_i$ is one of $\theta_0, \theta_1, \theta_2, \ldots, \theta_N$.
 By Lemma \ref{lem:eigvalW}, these are the  eigenvalues of $A$ on the irreducible $T$-module with endpoint $0$ and diameter $N$. This $T$-module exists because
$(0,N) \in \Psi$.
If $i \not\in \mathbb Z$ then $i$ is one of $1/2, 3/2, \ldots, N-1/2$. Therefore $\theta_i$ is one of
$\theta_{1/2}, \theta_{3/2}, \ldots, \theta_{N-1/2}$.
 By Lemma \ref{lem:eigvalW}, these are the eigenvalues of $A$ on an irreducible $T$-module with endpoint $1$ and diameter $N-1$. This $T$-module exists because $(1,N-1) \in \Psi$.
 We have shown that the eigenvalues of $A$ are  $\lbrace \theta_i \rbrace_{i\in \lbrack N \rbrack}$.
\end{proof}

\begin{definition} \label{def:Ei} \rm Let $\bf M$ denote the subalgebra of ${\rm Mat}_X(\mathbb C)$ generated by $A$.
For $i \in \lbrack N \rbrack$ let $E_i$ denote the primitive idempotent of $A$ for the eigenvalue $\theta_i$. 
\end{definition}
\begin{lemma} \label{lem:Mbasis} The matrices $\lbrace E_i \rbrace_{i \in \lbrack N \rbrack}$ form a basis for the vector space $\bf M$. Moreover
\begin{align*}
I = \sum_{i \in \lbrack N \rbrack} E_i, \qquad \quad E_i E_j = \delta_{i,j} E_i \;\;\bigl(i,j \in \lbrack N \rbrack \bigr), \qquad \quad A = \sum_{i \in \lbrack N \rbrack} \theta_i E_i.
\end{align*}
\end{lemma}
\begin{proof} By the linear algebra discussion in \cite[pp.~5,~6]{nomSpinModel}.
\end{proof}

\begin{corollary} The dimension of $\bf M$ is $2N+1$.
\end{corollary}
\begin{proof} By Lemma \ref{lem:Mbasis} and the comment below Definition \ref{def:thi}.
\end{proof}

\noindent Recall the standard module $V$. By construction, for $i \in \lbrack N \rbrack$ the subspace $E_iV$ is the eigenspace of $A$ for the eigenvalue $\theta_i$. 
 Our next general goal is to compute the dimension of $E_iV$. We will make use of the matrix $S$ from Definition \ref{def:Smat}.
\medskip

\begin{lemma} \label{lem:ESE} Let $i,j \in \lbrack N \rbrack$ such that $i+j \not=N$. Then $E_i S E_j=0$.
\end{lemma} 
\begin{proof} We have $AS+SA=0$ by Lemma  \ref{lem:SAAS}.
Observe that
\begin{align*}
0 = E_i (AS+SA) E_j = E_i S E_j (\theta_i + \theta_j ).
\end{align*}
We assume that $i + j \not=N$, so $\theta_i + \theta_j\not=0$ in view of  Lemma   \ref{lem:3step}(iii). Therefore $E_iSE_j=0$.
\end{proof}

\begin{lemma} \label{lem:SE} For $i \in \lbrack N \rbrack$ we have
\begin{enumerate}
\item[\rm (i)] $S E_i = E_{N-i}S$;
\item[\rm (ii)]  $S E_iV = E_{N-i}V$.
\end{enumerate}
\end{lemma}
\begin{proof} (i) Using Lemma \ref{lem:ESE} and $I = \sum_{j \in \lbrack N \rbrack} E_j$ we obtain
\begin{align*}
S E_i = \sum_{j \in \lbrack N \rbrack} E_j S E_i = E_{N-i} S E_i = \sum_{j \in \lbrack N \rbrack} E_{N-i}S E_j = E_{N-i}S.
\end{align*}
\noindent (ii) We have $SV=V$ since $S$ is invertible. By this and (i),
\begin{align*}
S E_i V = E_{N-i}S V = E_{N-i}V.
\end{align*}
\end{proof}

\begin{lemma} \label{lem:firstHalf} For $i \in \lbrack N \rbrack$ the eigenspaces $E_iV$ and $E_{N-i}V$ have the same dimension.
\end{lemma}
\begin{proof} By Lemma \ref{lem:SE}(ii) and since $S$ is invertible.
\end{proof}

\begin{lemma} \label{lem:dimSum} Let $i \in \lbrack N \rbrack$ such that $i \leq N/2$. Then
\begin{align*}
{\rm dim}\, E_iV = \sum_{r,d} {\rm mult} (r,d),
\end{align*}
where the sum is over the integers $r,d$ such that $(r,d) \in \Psi$ and $i -(N-d)/2 \in \mathbb N$.
\end{lemma}
\begin{proof} 
Write $V$ as an orthogonal direct sum of irreducible $T$-modules. The set of irreducible $T$-modules in this sum will be denoted by $\Lambda$. We have
\begin{align*}
V=\sum_{W \in \Lambda} W.
\end{align*}
In the above equation we apply $E_i$ to
each term, and get an orthogonal direct sum $E_iV = \sum_{W\in \Lambda} E_iW$. For each $W \in \Lambda$ we consider the dimension of $E_iW$.
Let $d$ denote the diameter of $W$, and write $t=(N-d)/2$. By Lemma \ref{lem:eigvalW} we have
 \begin{align*}
{\rm dim} \,E_iW = \begin{cases}  
1, & {\mbox{\rm if $i-t \in \mathbb N$}};\\
0, & {\mbox{\rm if $i - t \not\in \mathbb N$}}.
\end{cases}
\end{align*}
The result follows.
\end{proof}

\noindent  Recall the scalars 
 $\mu_r$ from  \eqref{eq:muForm}.

\begin{theorem} \label{thm:dimEV} For $i \in \lbrack N \rbrack$ and $i \leq N/2$, the dimension of $E_iV$ is given below.
\begin{enumerate}
\item[\rm (i)] If $i \in \mathbb Z$, then the dimension is
\begin{align*}
\mu_0 \binom{N}{i}_q + \mu_2 \binom{N-2}{i-1}_q + \mu_4 \binom{N-4}{i-2}_q  + \cdots + \mu_{2i} \binom{N-2i}{0}_q.
\end{align*}
\item[\rm (ii)] If $i\not\in \mathbb Z$,  then the dimension  is
\begin{align*}
\mu_1 \binom{N-1}{i-1/2}_q + \mu_3 \binom{N-3}{i-3/2}_q + \mu_5 \binom{N-5}{i-5/2}_q + \cdots + \mu_{2i} \binom{N-2i}{0}_q.
\end{align*}
\end{enumerate}
\end{theorem}
\begin{proof}  By Proposition  \ref{prop:PsiSum} and Lemma \ref{lem:dimSum}.
\end{proof}

%
 %
 %
 
\section{A $Q$-polynomial structure for the poset $\mathcal A_q(N,M)$}

\noindent We continue to discuss the Attenuated Space poset $\mathcal A_q(N,M)$. Recall the diagonal
matrix $A^*$ from Definition  \ref{def:As},  and the $q$-adjacency matrix $A$ from Definition 
\ref{def:A}. Lemma \ref{lem:AsTD} shows that $A$ acts on the eigenspaces of $A^*$ in a (block) tridiagonal fashion.
In this section, 
we show that $A^*$ acts on the eigenspaces of $A$ in a (block) tridiagonal fashion.
 We use this action to 
show that $A$ is $Q$-polynomial.
\medskip

\noindent The next two propositions show that for distinct $i,j \in \lbrack N \rbrack$,
\begin{equation}
\mbox{\rm $E_i A^* E_j \not=0$ \quad if and only if \quad $\vert i-j \vert =1$.} \label{eq:iff}
 \end{equation}

 
 \begin{proposition} \label{prop:A3} Let $i,j \in \lbrack N \rbrack$ such that $i - j \not\in \lbrace 1,0,-1\rbrace$. Then 
 $E_i A^* E_j = 0$.
 \end{proposition}
 \begin{proof}   Recall the abbreviation $\beta = q + q^{-1}$.
 We invoke the first equation in Proposition \ref{prop:TD1}. Referring to that equation, let $\Delta$ denote the left-hand side
 minus the right-hand side. We have $\Delta=0$, so $E_i \Delta E_j=0$. We evaluate $E_i \Delta E_j$ using
 $E_i A = \theta_i E_i$ and $A E_j = \theta_j E_j$. This yields
 \begin{align*}
 0 & = E_i \Delta E_j \\
    &= E_i A^* E_j \biggl( \theta^3_i - (\beta+1) \theta^2_i \theta_j + (\beta+1) \theta_i \theta^2_j - \theta^3_j - q^{N+M-2} (q+1)^2 (\theta_i - \theta_j)\biggr) \\
    &= E_i A^* E_j \bigl(\theta_i - \theta_j \bigr) \Bigl( \theta^2_i - \beta \theta_i \theta_j + \theta^2_j - q^{N+M-2} (q+1)^2 \Bigr) \\
      &= E_i A^* E_j \bigl(\theta_i - \theta_j \bigr)\frac{\bigl( q^i-q^{j-1}\bigr)  \bigl( q^i-q^{j+1}\bigr)  \bigl( q^{N-i-j-1}+1\bigr)  \bigl( q^{N-i-j+1}+1\bigr) q^M}{(q-1)^2}.
 \end{align*}
 We examine the coefficient of $E_i A^* E_j$ in the previous line. We show that this coefficient is nonzero. The  factor $\theta_i - \theta_j$ is nonzero
 because $i \not=j$. To the right of $\theta_i - \theta_j$ we see a fraction with  five  factors in the numerator. The five factors  are nonzero, because $i\not= j\pm 1$ and since $q$ is nonzero and not a root of unity. 
 We have shown that the coefficient of $E_iA^*E_j$ is nonzero. Therefore $E_i A^* E_j =0$. 
 \end{proof}
 
 \begin{proposition}\label{prop:TwoPath} Let $i,j \in \lbrack N \rbrack$ such that $\vert i - j \vert = 1$. Then $E_i A^* E_j \not=0$.
 \end{proposition}
 \begin{proof} Recall the matrix $S$ from Definition \ref{def:Smat}. Note that 
 \begin{align*}
 S (E_i A^* E_j) S^{-1} = (S E_i S^{-1})( S A^* S^{-1})( S E_j S^{-1}) =
 E_{N-i} A^* E_{N-j}. 
 \end{align*}
\noindent Therefore, $E_i A^* E_j = 0 $ if and only if $E_{N-i} A^* E_{N-j} = 0$. Replacing the pair $(i,j)$ by $(N-i,N-j)$ if necessary, we may assume that $i-j=1$.
 Our strategy is to assume $E_i A^* E_j=0$, and get a contradiction. For notational convenience, define
 \begin{align*}
\tau = \begin{cases}  
0, & {\mbox{\rm if $i,j \in \mathbb Z$}};\\
1/2, & {\mbox{\rm if $i,j \not\in \mathbb Z$}}.
\end{cases}
\end{align*}
 \noindent Define
 \begin{align*}
 W = E_jV + E_{j-1}V + \cdots + E_\tau V.
 \end{align*}
 We have $E_iW=0$ since $i=j+1$. We claim that $W$ is a $T$-module. To prove the claim, note that $AW \subseteq W$.
We have $A^* W \subseteq W$ by $E_iA^*E_j=0$ and Proposition  \ref{prop:A3}. The matrices $A, A^*$ generate $T$, so 
$W$ is a $T$-module as claimed.  Let $U$ denote an irreducible $T$-module with endpoint $2\tau$ and diameter $N-2\tau $. The
eigenvalues of $A$ on $U$ are $\lbrace \theta_{\tau+\ell} \rbrace_{\ell =0}^{N-2\tau}$. Therefore $E_{\tau+\ell} U \not=0$ for $0 \leq \ell \leq N-2\tau$.
In particular $E_\tau U\not=0$. By construction $E_\tau U \subseteq U \cap W$, so  $U \cap W \not=0$. Both $U$ and $W$ are $T$-modules, so $U\cap W$ is a $T$-module.
The $T$-module $U$ is irreducible, so $U\subseteq W$.
We have $0 \not= E_iU \subseteq E_iW$, for a contradiction. We have shown that $E_i A^* E_j \not=0$.
 \end{proof}
 
 \noindent The concept of a dual adjacency matrix is explained in \cite[Definition~2.2]{Lnq}. We will use this concept to describe how $A^*$ acts on the eigenspaces of $A$.
 Recall the algebra ${\bf M^*}$ from above Definition \ref{def:As}.
 For notational convenience, define $D=2N$. Let $\lbrace V_i \rbrace_{i=0}^{D}$ denote an ordering of the eigenspaces of $A$. As defined in \cite[Definition~2.2]{Lnq}, a dual adjacency matrix for $\lbrace V_i \rbrace_{i=0}^{D}$
 is a matrix $\bf A^*$ that generates ${\bf M^*}$ and
 \begin{align*}
 {\bf A^*} V_i \subseteq V_{i-1} + V_i + V_{i+1} \qquad \qquad (0 \leq i \leq D),
 \end{align*}
 \noindent where $V_{-1}=0$ and $V_{D+1}=0$.

 \begin{corollary} \label{cor:maincor}
 For the poset $\mathcal A_q(N,M)$ the matrix $A^*$ is a dual adjacency matrix  for the following  orderings of the eigenspaces of $A$: 
\begin{enumerate}
\item[\rm (i)]    
$ E_0V< E_1V< E_2V < \cdots < E_NV< E_{1/2}V < E_{3/2}V < \cdots < E_{N-1/2}V $;
\item[\rm (ii)] 
 $E_{1/2}V < E_{3/2}V < \cdots < E_{N-1/2}V < E_0V<  E_1V< E_2 V< \cdots < E_NV$.
 \end{enumerate}
 \end{corollary}
 \begin{proof}  By Lemma \ref{lem:AsE}  and
 Propositions \ref{prop:A3}, \ref{prop:TwoPath}.
 \end{proof}

\noindent An ordering of the eigenspaces of $A$ is called {\it $Q$-polyomial} whenever it has a dual adjacency matrix \cite[Definition~2.3]{Lnq}.
 
 \begin{corollary} \label{cor:QpolyA} For the poset $\mathcal A_q(N,M)$ the  orderings {\rm (i), (ii)} in Corollary \ref{cor:maincor} are $Q$-polynomial.
 \end{corollary}
 \begin{proof} By Corollary \ref{cor:maincor}.
 \end{proof}
 \noindent The matrix $A$ is said to be {\it $Q$-polynomial} whenever there exists a $Q$-polynomial ordering of its eigenspaces \cite[Definition~2.4]{Lnq}.
 \begin{theorem} \label{thm:mainRes} For the poset $\mathcal A_q(N,M)$ the $q$-adjacency matrix $A$ is $Q$-polynomial.
 \end{theorem} 
 \begin{proof} By Corollary \ref{cor:QpolyA}.
 \end{proof}

 \section{The irreducible $T$-modules and Leonard pairs}
 We continue to discuss the Attenuated Space poset $\mathcal A_q(N,M)$. 
 Recall the subconstituent algebra $T$ from Definition \ref{def:SubT}.
 In this section, we describe the irreducible $T$-modules from the point of view of Leonard pairs \cite{ter2}.
 \medskip
 
 \noindent
 The definition of a Leonard pair and Leonard system can be found in  \cite[Definition~1.1]{ter2} and \cite[Definition~1.4]{ter2}, respectively.
 We refer the reader to \cite{bbit,nomKraw, nomTB, introLP, terPA, LPqRac, terTDD, aa,LSnotes, vidunas}
 for the basic facts about these concepts.
 \medskip
 
 \noindent
Let $W$ denote an irreducible $T$-module, with endpoint $r$ and diameter $d$. Recall from Lemma \ref{lem:Iso1} that
\begin{align*}
0 \leq r,d\leq N, \qquad N - 2r \leq d \leq N-r, \qquad d \leq N+M-2r.
\end{align*}
Recall the parameter $t=(N-d)/2$.
In Lemmas  \ref{lem:AsW}, \ref{prop:AonW} 
we described the actions of  $A, A^*$ on $W$. Comparing these actions with  \cite[Example~20.7]{LSnotes}, we find that 
the sequence $(A; \lbrace E_{t+i} \rbrace_{i=0}^{d}; A^*; \lbrace E^*_{r+i} \rbrace_{i=0}^{d})$  acts on $W$ as a Leonard system of dual $q$-Krawtchouk type.
For notational convenience, let $\Phi$ denote this Leonard system. By Lemma  \ref{lem:AsTD} we see that $\Phi$  is bipartite
in the sense of \cite[Definition~5.1]{hanson}. By Lemmas  \ref{lem:AsW}, \ref{lem:eigvalW} and line \eqref{eq:iff},  the eigenvalue sequence  and dual eigenvalue sequence
of $\Phi$ are
\begin{align*}
\theta_i(\Phi) = \frac{q^{d-i}-q^i}{q-1}\,q^{M/2+t}, \qquad \qquad \theta^*_i(\Phi) = q^{-r-i} \qquad \qquad (0 \leq i \leq d).
\end{align*}
Using the data in \cite[Example~20.7]{LSnotes}
we find that the Leonard system $\Phi$  has parameters
\begin{align*}
&d(\Phi) = d, \qquad \qquad h(\Phi) = \frac{q^{d+t+M/2}}{q-1}, \qquad \quad h^*(\Phi)=q^{-r}, \\
&s(\Phi) = -q^{-d-1}, 
\qquad \qquad \theta_0(\Phi) = \frac{q^d-1}{q-1}\,q^{M/2+t}, \qquad \qquad \theta^*_0(\Phi) = q^{-r}.
\end{align*}
\noindent We close with a few remarks. The bipartite property is discussed in \cite{caughman1, caughman2, caughman3, nomTB}.
The Leonard systems of dual $q$-Krawtchouk type are discussed in \cite{nomNB, worDPG}.
 The Leonard systems of dual $q$-Krawtchouk type are related to the orthogonal polynomial systems of dual $q$-Krawtchouk type \cite{terPA,aa}.

\section{Acknowledgement}
The author thanks Wen Liu, \v{S}tefko Miklavi\v{c}, Kazumasa Nomura, and Murali Srinivasan for valuable comments and suggestions about the paper.


\bigskip

\noindent Paul Terwilliger \hfil\break
\noindent Department of Mathematics \hfil\break
\noindent University of Wisconsin \hfil\break
\noindent 480 Lincoln Drive \hfil\break
\noindent Madison, WI 53706-1388 USA \hfil\break
\noindent email: {\tt terwilli@math.wisc.edu }\hfil\break

\section{Statements and Declarations}

\noindent {\bf Funding}: The author declares that no funds, grants, or other support were received during the preparation of this manuscript.
\medskip

\noindent  {\bf Competing interests}:  The author  has no relevant financial or non-financial interests to disclose.
\medskip

\noindent {\bf Data availability}: All data generated or analyzed during this study are included in this published article.

\end{document}